\documentclass{amsart}
\title{Operators of rank $1$, discrete path integration and graph 
Laplacians} 
\author[Yu.Burman]{Yurii Burman} 
\address{Independent University of Moscow and Higher School of Economics, 
Moscow} 
\email{burman@mccme.ru} 
\date{} 
\usepackage{emlines}

\makeatletter

\RequirePackage{amssymb}

\@addtoreset{theorem}{section}

\newcommand{\theoremName}{Theorem}
\newcommand{\lemmaName}{Lemma}
\newcommand{\corollaryName}{Corollary}
\newcommand{\statementName}{Proposition}

\newcommand{\remarkName}{Remark}
\newcommand{\exampleName}{Example}
\newcommand{\definitionName}{Definition}
\newcommand{\problemName}{Problem}
\newcommand{\proofName}{Proof}
\renewcommand{\proofname}{\proofName}
\newcommand{\answerName}{Answer}
\newcommand{\hintName}{Hint}

\theoremstyle{plain}

\newtheorem {theorem}{\theoremName}
\newtheorem {lemma}[theorem]{\lemmaName}
\newtheorem {corollary}[theorem]{\corollaryName}

\theoremstyle{remark}

\newtheorem{remark}[theorem]{\remarkName}
\newtheorem{Remark}{\remarkName}

\newtheorem{example}[theorem]{\exampleName}

\theoremstyle{definition}

\newtheorem{Definition}{\definitionName}

\let\@newpf\proof 
\let\proof\relax

\def \namepf[#1] {\@newpf[\proofname\ #1]}
\newenvironment{proof}{\@ifnextchar[{\namepf}{\@newpf[\proofname]}}{\qed\endtrivlist}

\newcounter{qst}
\@addtoreset{qst}{problem}

\def \Real {{\mathbb R}}

\def \lnorm#1\rnorm {\vphantom{#1}\left\|\smash{#1}\right\|}
\def \lmod#1\rmod {\vphantom{#1}\left|\smash{#1}\right|}

\newcommand \bydef {\stackrel{\mbox{\scriptsize def}}{=}}

\@ifundefined{name}{\newcommand \name[1] {\mathop{\rm #1}\nolimits}}%
{\relax}

\newcommand \eps {\varepsilon}
\renewcommand \phi {\varphi}
\renewcommand \rho {\varrho}
\renewcommand \emptyset {\varnothing}

\makeatother

\begin{document}

 \begin{abstract}
We prove a formula for a characteristic polynomial of an operator expressed 
as a polynomial of rank $1$ operators. The formula uses a discrete analog 
of path integration and implies a generalization of the Forman--Kenyon's 
formula for a determinant of the graph Laplacian \cite{Kenyon,Forman} 
(which, in its turn, implies the famous matrix-tree theorem by Kirchhoff) 
as well as its level $2$ analog, where the summation is performed over 
triangulated nodal surfaces with boundary.
 \end{abstract}

\maketitle

\section*{Introduction}

The primary impulse to write this article was the famous Matrix-tree 
theorem (MTT), first discovered by Kirchhoff \cite{Kirchhoff} in 1847 and 
re-proved more than a dozen times since then. The theorem in its simplest 
form expresses the number of spanning trees in a graph as a determinant of 
a suitably chosen matrix $M$. See \cite{Chaiken} for a review of existing 
proofs and generalizations. One of the proofs, given in \cite{BPT}, makes 
use of the fact that the matrix in question is a weighted sum of special 
rank $1$ matrices (identity minus reflection). This structure of $M$ 
explains why the MTT appeares sometimes in quite unexpected contexts (see 
\cite{BZ} for just an example).

Let $M$ be an operator expressed explicitly as a noncommutative polynomial 
of arbitrary operators $M_i$ of rank $1$. The article contains a formula 
(Theorem \ref{Th:Main}) for the characterstic polynomial of $M$ in terms of 
$M_i$. Corollaries of the formula include the MTT (cf,\ Corollary 
\ref{Cr:MTT}), the $D$-analog of the MTT (\cite{BPT}, cf.\ Corollary 
\ref{Cr:MTTD}), a formula for the determinant of the Laplacian of a line 
bundle on a graph \cite{Forman,Kenyon} (Section \ref{SSSec:Lapl}; actually, 
we compute the whole characteristic polynomial), and a level $2$ analog of 
the formula from \cite{Forman} proved in Sections \ref{SSSec:Level2} and 
\ref{SSec:Level2}. 

The right-hand side of the main formula \eqref{Eq:Main} involve summation 
over the graphs consisting of several cycles and/or several chains. The 
summand (the weight $W_P$, see \eqref{Eq:DefWeight}) is a function on edges 
of the graph obtained also by symmation over the set of paths joining the 
endpoints of the edge. Consequently, corollaries of the main theorem 
involve summation over various objects, including trees (the MTT), 
hypertrees (the Massbaum--Vaintrob theorem), cycle-rooted trees (the 
Forman--Kenyon formula) and nodal surfaces with boundary (the level $2$ 
analog from Section \ref{SSSec:Level2}). The standard expression of a 
determinant via summation over the permutation group also follows from 
Theorem \ref{Th:Main} --- see Section \ref{SSSec:Det}.

\subsubsection*{Structure of the article} Section \ref{SSec:Main} contains 
the formulation and the proof of the main theorem (Theorem \ref{Th:Main}). 
Section \ref{SSec:Ex} lists several corollaries of the theorem. Some 
corollaries are proved immediately, proofs of some others require not some 
additional reasoning, which is given in Section \ref{Sec:Proofs}. Section 
\ref{SSec:Tech} contains technical lemmas (including a duality lemma 
\ref{Lm:OrthComp} for the angle between two subspaces of a Euclidean 
space). Section \ref{SSec:MTree} contains the proof of the generalization 
of Forman's formula (Theorem \ref{Th:Forman}), and Section 
\ref{SSec:Level2} is devoted to the proof of Theorem \ref{Th:DetLevel2} 
about the level $2$ analog of the graph Laplacian.

\subsubsection*{Open questions and future research} Generalizations of the 
matrix-tree theorem are plentiful and versatile (see e.g.\ \cite{Chaiken} 
and the references therein), and Theorem \ref{Th:Main} covers surprisingly 
many of them. Nevertheless, there are numerous results in the field whose 
relations with the discrete path integration technique are yet to be 
clarified. One of such results is the Hyperpfaffian-cactus theorem by 
A.Abdessalam \cite{Abdes}; it is a generalization of the Pfaffian Hypertree 
theorem of \cite{MasVai}. The latter theorem is not included into this 
paper but was proved in \cite{BPT} by a method close to Theorem 
\ref{Th:Main} (the original proof due to G.Massbaum and A.Vaintrob was 
quite different); we were unable, though, to extend this proof to 
hyperpfaffians. Also, some results of the MTT type appear in the theory of 
determinantal point processes, see \cite{BurPem} and \cite{Borodin}. It 
would be interesting to know whether these results are covered by Theorem 
\ref{Th:Main} or by its suitable generalization.

The summation in Theorem \ref{Th:Main} is performed over the set of graphs 
that deserve to be called ``discrete $1$-manifolds'' (oriented, possibly 
with boundary). The summand is obtained by a procedure which is quite 
natural to understand as a ``discrete path integration''. A tempting 
direction of the future research here is to send dimension to infinity, 
getting ``real'' path integration and summation over $2$-varieties. The 
author plans to write a special paper about this. 

\section{Sums of operators of rank $1$}

\subsection{The main theorem}\label{SSec:Main}

Let $V$ be a vector space of dimension $n$ with a scalar product $\langle 
\cdot, \cdot\rangle$. Choose an integer $N$ and fix two sequences of 
vectors, $e_1, \dots, e_N \in V$ and $\alpha_1, \dots, \alpha_N \in V$. For 
any $i$ denote by $M_i$ the operator given by $M_i(v) = \langle 
\alpha_i,v\rangle e_i$; one has $\name{rk} M_i = 1$ or $M_i = 0$. Consider 
an operator 
 \begin{equation}\label{Eq:DefM}
M = P(M_1, \dots, M_N)
 \end{equation}
where
 \begin{equation*}
P(x_1, \dots, x_N) = \sum_{s=1}^k \sum_{1 \le i_1, \dots, i_s \le N} 
p_{i_1, \dots, i_s} x_{i_1} \dots x_{i_s}
 \end{equation*}
is a noncommutative polynomial of degree $k$. This section contains a 
description of the characteristic polynomial $\name{char}_M(t)$ of the 
operator $M$.

Let $G$ be a finite graph with the vertices $1, 2, \dots, N$; let $a$ and 
$b$ be its vertices. Define the weight $W_P(a,b)$ by
 \begin{equation}\label{Eq:DefWeight}
W_P(a,b) = \sum_{s=1}^k \sum_{i_1=a, i_2, \dots, i_s=b} p_{i_1, \dots, i_s} 
\langle \alpha_{i_2}, e_{i_1}\rangle \langle \alpha_{i_3}, e_{i_2}\rangle 
\dots \langle \alpha_{i_s}, e_{i_{s-1}}\rangle,
 \end{equation}
(the internal summation is taken over the set of paths $i_1, \dots, i_s$ of 
length $s$ joining the vertices $a = i_1$ and $b = i_s$). Also, denote
 \begin{equation}\label{Eq:Weight}
W_P(G) \bydef \prod_{\text{\scriptsize $(a,b)$ is an edge of $G$}} W_P(a,b)
 \end{equation}

A directed graph $G$ with the vertices $1, 2, \dots, N$ is called a 
discrete oriented one-dimensional manifold with (possibly empty) boundary 
(abbreviated as DOOMB) if every its connected component is either an 
oriented chain (a graph with $\ell$ distinct vectices $i_1, \dots, 
i_\ell$ and the edges are $(i_1, i_2)$, $(i_2, i_3)$, \dots, $(i_{\ell-1}, 
i_\ell)$) or an oriented cycle (the same edges but the vertices are $i_1, 
\dots, i_{\ell-1}$ and $i_\ell = i_1$).

 \begin{theorem} \label{Th:Main}
$\name{char}_M(t) = \sum_{k=0}^n (-1)^k \mu_k t^{n+1-k}$ where 
 \begin{equation}\label{Eq:Main}
\mu_k = \sum_{G \in \name{\mathcal D}_{n,k}} W_P(G) \det 
(\langle \alpha_{d_1^-}, e_{d_2^+}\rangle).
 \end{equation}
Here $\name{\mathcal D}_{n,k}$ is the set of DOOMBs with the vertices $1, 
2, \dots, n$ and $k$ edges; $d_1$ and $d_2$ run through the set of all 
edges of the graph $G$; $d^-$ and $d^+$ are the initial and the terminal 
vertex of the directed edge $d$. 
 \end{theorem}

 \begin{Remark}
The main theorem of \cite{BPT} is a particular case of Theorem 
\ref{Th:Main}.
 \end{Remark}

 \begin{proof}
Consider an orthonormal basis $u_1, \dots, u_n \in V$ and fix a sequence 
$j_1, \dots, j_k$, $1 \le j_1 < \dots < j_k \le N$. Then
 \begin{align*}
M(u_{j_1} \wedge \dots \wedge u_{j_k}) &= \sum_{s_1, \dots, s_k} 
\!\! \sum_{\begin{array}{c}
\scriptstyle 1 \le i_m^{(q)} \le N\\
\scriptstyle 1 \le m \le s_q, 1 \le q \le k
\end{array}} 
\prod_{q=1}^k \left(p_{i_1^{(q)}, \dots, i_{s_q}^{(q)}} 
\prod_{r=2}^{s_q}\langle \alpha_{i_r^{(q)}}, e_{i_{r-1}^{(q)}}\rangle 
\right) \\
&\hphantom{\sum_{s_1, \dots, s_k}}\times \prod_{q=1}^k \langle 
\alpha_{i_1^{(q)}}, u_{j_q}\rangle \times e_{i_{s_1}^{(1)}} \wedge \dots 
\wedge e_{i_{s_k}^{(k)}}\\
&= \sum_{s_1, \dots, s_k} 
\!\! \sum_{\begin{array}{c}
\scriptstyle 1 \le i_m^{(q)} \le N\\
\scriptstyle 1 \le m \le s_q, 1 \le q \le k,\\
\scriptstyle i_1^{(1)} < \dots < i_1^{(k)} 
\end{array}}
\prod_{q=1}^k \left(p_{i_1^{(q)}, \dots, i_{s_q}^{(q)}} 
\prod_{r=2}^{s_q}\langle \alpha_{i_r^{(q)}}, e_{i_{r-1}^{(q)}}\rangle 
\right) \\
&\hphantom{\sum_{s_1, \dots, s_k}}\times \sum_{\sigma \in S_k} 
(-1)^\sigma \prod_{q=1}^k \langle \alpha_{i_1^{(q)}}, 
u_{j_{\sigma(q)}}\rangle \times e_{i_{s_1}^{(1)}} \wedge \dots \wedge 
e_{i_{s_k}^{(k)}}
 \end{align*}
The coefficient at $u_{j_1} \wedge \dots \wedge u_{j_k}$ in $M(u_{j_1} 
\wedge \dots \wedge u_{j_k})$ is then equal to
 \begin{align*}
\sum_{s_1, \dots, s_k} 
\!\! &\sum_{\begin{array}{c}
\scriptstyle 1 \le i_m^{(q)} \le N\\
\scriptstyle 1 \le m \le s_q, 1 \le q \le k,\\
\scriptstyle i_1^{(1)} < \dots < i_1^{(k)} 
\end{array}}
\prod_{q=1}^k \left(p_{i_1^{(q)}, \dots, i_{s_q}^{(q)}} 
\prod_{r=2}^{s_q}\langle \alpha_{i_r^{(q)}}, e_{i_{r-1}^{(q)}}\rangle 
\right) \\
&\times \det (\langle \alpha_{i_1^{(q)}}, u_{j_r}\rangle)_{1 \le q, r \le 
k} \times \det (\langle u_{j_r}, e_{i_{s_q}^{(q)}}\rangle)_{1 
\le q, r \le k}.
 \end{align*}
Hence
 \begin{align}
\mu_k &= \name{Tr} M^{\wedge k} = \sum_{s_1, \dots, s_k} 
\!\! \sum_{\begin{array}{c}
\scriptstyle 1 \le i_m^{(q)} \le N\\
\scriptstyle 1 \le m \le s_q, 1 \le q \le k,\\
\scriptstyle i_1^{(1)} < \dots < i_1^{(k)} 
\end{array}}
\prod_{q=1}^k \left(p_{i_1^{(q)}, \dots, i_{s_q}^{(q)}} 
\prod_{r=2}^{s_q}\langle \alpha_{i_r^{(q)}}, e_{i_{r-1}^{(q)}}\rangle 
\right) \label{Eq:Trace}\\
&\hphantom{\name{Tr} M^{\wedge k} = \sum_{s_1, \dots, s_k}}
\times \det (\langle \alpha_{i_1^{(r)}}, e_{i_{s_q}^{(q)}}\rangle)_{1 \le 
q,r \le k}\nonumber
 \end{align}

A multi-index $i_1^{(q)}, \dots, i_{s_q}^{(q)}$, $1 \le q \le k$, can be 
interpreted as a directed graph $G$ with the edges $(i_1^{(q)}, 
i_{s_q}^{(q)})$, $1 \le q \le k$, and a path $i_1^{(q)}, \dots, 
i_{s_q}^{(q)}$ joining endpoints $d_- = i_1^{(q)}$ and $d_+ = 
i_{s_q}^{(q)}$ of every edge. Conversely, a graph plus collection of paths 
is just the multi-index. So one can rearrange summation in \eqref{Eq:Trace} 
to sum over the graphs $G$ first, and then over the set of all paths for a 
given graph, getting $W(G)$ times the determinant. The determinant may be 
nonzero only if no two $i_1^{(q)}$ and no two $i_{s_q}^{(q)}$ are equal. 
Thus, in a graph $G$ entering the sum every vertex has at most one outgoing 
edge and at most one incoming edge. This means that $G$ is a DOOMB.
 \end{proof}

\subsection{Graph Laplacians and other examples}\label{SSec:Ex}

Here are some corollaries of Theorem \ref{Th:Main}.

\subsubsection{Determinants}\label{SSSec:Det}

It's a funny result demonstrating the nature of Theorem \ref{Th:Main}. 

Let $u_1, \dots, u_n$ be an orthonormal basis in $\Real^n$. Take $e_{ij} 
\bydef = u_i$ and $\alpha_{ij} \bydef = u_j$ for all $1 \le i,j \le n$. 
Take 
 \begin{equation*}
P(x_{11}, \dots, x_{nn}) \bydef \sum_{i,j=1}^n a_{ij} x_{ij}
 \end{equation*}
and apply Theorem \ref{Th:Main}. The matrix of the opertor $M$ in the basis 
$u_1,\dots, u_n$ is $(a_{ij})$. The polynomial $P$ is linear, so the 
paths entering equation \eqref{Eq:DefWeight} all lave length $1$. 
Consequently, the DOOMB $G$ of \eqref{Eq:Main} must be a union of $k$ loops 
attached to the vertices $(i_1,j_1), \dots, (i_k,j_k)$. So the summation in 
\eqref{Eq:Main} is over the set of unordered $k$-tuples $(i_1,j_1), \dots, 
(i_k,j_k)$ with $1 \le i_s, j_s \le N$ for all $s$. In other words, the 
summation is over the set of graphs $F$ with the vertices $1, 2, \dots, n$ 
and $k$ unnumbered directed edges (loops are allowed).

One has $\langle \alpha_{ij},e_{kl}\rangle = \delta_{jk}$, so the 
contribution of a graph $F$ into \eqref{Eq:Main} is equal to $a_{i_1 j_1} 
\dots a_{i_k j_k} \det (\delta_{i_p j_q})_{1 \le p \le k}^{1 \le q \le k}$. 
It is easy to see that the determinant is nonzero only if all the $i_p$ and 
all the $j_q$ are distinct (else the matrix has identical rows or columns), 
and for every $q$ there is $p \bydef \sigma(q)$ such that $j_q = i_p$ (else 
a matrix has a zero row). If these conditions are satisfied, the 
determinant is obviously $(-1)^{\name{sgn}(\sigma)}$ where 
$\name{sgn}(\sigma)$ is the parity of the permutation $\sigma$. Hence, 
Theorem \ref{Th:Main} in this case is reduced to the usual formula 
expressing coefficients of the characteristic polynomial of the operator 
via its matrix elements. 

\subsubsection{Graph Laplacians}\label{SSSec:Lapl}

Let $F$ be a directed graph without loops. Following \cite{Kenyon}, give 
the following definition:

 \begin{Definition}
A {\em line bundle with connection on $F$} is a function attaching a number 
$\phi_e \ne 0$ to every directed edge $e$ of the graph. By definition, also 
take $\phi_{-e} = \phi_e^{-1}$ where $-e$ is the edge $e$ with the 
direction reversed.
 \end{Definition}

To explain the name, attach a one-dimensional space $\Real$ (a fiber of the 
bundle) to every vertex of $F$ and interpret the number $\phi_e$ as the $1 
\times 1$-matrix of the operator of parallel transport along the edge $e$. 
For a path $\Lambda = (e_1, \dots, e_k)$ denote $\phi_\Lambda \bydef 
\phi_{e_1} \dots \phi_{e_k}$ (the operator of parallel transport along 
$\Lambda$). If the path $\Lambda$ is a cycle then $\phi_\Lambda$ is called a 
holonomy of the cycle.

Suppose now that $F$ has the vertices $1, 2, \dots, n$ and no multiple 
edges. Supply also every edge $(i,j)$ of $H$ with a weight $c_{ij} = 
c_{ji}$. Take $N = n(n-1)/2$, 
 \begin{equation}\label{Eq:DefPLapl}
P(x_{12}, \dots, x_{n-1,n}) = \sum_{1 \le i < j \le n} c_{ij} x_{ij},
 \end{equation}
and $e_{ij} \bydef u_i - \phi_{ij} u_j$ and $\alpha_{ij} \bydef u_i - 
\phi_{ji} u_j)$, and consider the operator $M$ like in \eqref{Eq:DefM}.

If $v = \sum_{i=1}^n v_i u_i$ then 
 \begin{equation}\label{Eq:Lapl}
M(v) = \sum_{1 \le i < j \le n} c_{ij} (v_i - \phi_{ji}v_j) (u_i - 
\phi_{ij}u_j) = \sum_{i=1}^n u_i \sum_{j \ne i} c_{ij} (v_i - \phi_{ji} 
v_j).
 \end{equation}
The operator $M$ is called in \cite{Kenyon} a Laplacian of the bundle.

Call a graph $F$ a {\em mixed forest} if every its connected component is 
either a tree or a graph with one cycle (a connected graph with the number 
of vertices equal to the number of edges). The graphs where each component 
contains one cycle are called {\em cycle-rooted spanning forests (CRSF)} in 
\cite{Kenyon}; hence the name ``mixed forest'' here.

The following corollary of Theorem \ref{Th:Main} generalizes the 
Matrix-CRSF theorem of \cite{Forman} and \cite{Kenyon}:

 \begin{theorem} \label{Th:Forman}
The characteristic polynomial of the Laplacian \eqref{Eq:Lapl} of a line 
bundle on a graph is equal to $\sum_{k=0}^n (-1)^k \mu_k t^k$ where 
 \begin{equation}\label{Eq:MixForest}
\mu_k = \sum_{F \in \name{\mathcal{MF}}_{n,k,\ell}} \prod_{\substack{(pq) 
\text{ is} \\ \text{an edge of } F}} c_{pq} \prod_{i=1}^{n-k} (m_i+1) 
\prod_{j=1}^\ell (1-w_j)(1-1/w_j).
 \end{equation}
Here $\name{\mathcal{MF}}_{n,k,\ell}$ is the set of mixed forests 
containing $n$ vertices, $k$ edges and split into $n-k$ tree components and 
$\ell$ one-cycle componenets; $m_i$ is the number of edges in the $i$-th 
component, and $w_j$ is the holonomy of the cycle in the $j$-th component.
 \end{theorem}

A special case of Theorem \ref{Th:Forman} arises if $\phi_{ij}=1$ for all 
$i,j$. Then \eqref{Eq:Lapl} implies $M = \sum_{1 \le p < q \le n} c_{pq} 
(1-\sigma_{pq})$ where $\sigma_{pq}$ is a reflection exchanging the $p$-th 
and the $q$-th coordinate: $\sigma_p(u_p) = u_q$, $\sigma_{pq}(u_q) = u_p$ 
and $\sigma_{pq}(u_i) = u_i$ for $i \ne p,q$. The reflections $\sigma_{pq}$ 
generate the Coxeter group $A_n$. The holonomies here are all equal to $1$, 
so a mixed forest entering a summation in Theorem \ref{Th:Forman} cannot 
have cycles and should be a ``real'' forest:

 \begin{corollary} \label{Cr:MTT}
The characteristic polynomial of the operator $M = \sum_{1 \le p < q \le 
n-1} c_{pq} (1-\sigma_{pq})$ is equal to $\sum_{k=0}^n (-1)^k \mu_k t^k$ 
where 
 \begin{equation*}
\mu_k = \sum_{F \in \name{\mathcal{F}}_{n,k}} \prod_{\substack{(pq) \text{ 
is}\\ \text{an edge of } F}} c_{pq} \prod_{i=1}^{n-k} (m_i+1).
 \end{equation*}
Here $\name{\mathcal{F}}_{n,k}$ is the set of forests with $n$ vertices, 
$k$ edges and $n-k$ components; $m_i$ is the number of edges in the $i$-th 
component.
 \end{corollary}

Apparently, $\det M = 0$ (there are no forests with $n$ vertices and $n$ 
edges), so the summation is indeed up to $k = n-1$. This corollary follows 
also from the classical Principal Minors Matrix-Tree Theorem, see e.g.\ 
\cite{Chaiken} for proofs and related results.

Another possibility is to join every pair of vertices $(i,j)$ with {\em 
two} edges: $(i,j)_-$ with $\phi_{ij}^- = 1$ (``a $-$-edge'', because 
$e_{ij}^{-} = u_i - u_j$) and $(i,j)_+$ with $\phi_{ij}^+ = -1$ (``a 
$+$-edge'' because $e_{ij}^{+} = u_i + u_j$); the weights are $c_{ij}^{+}$ 
and $c_{ij}^{-}$, respectively. The holonomy of a cycle is $w = (-1)^d$ 
where $d$ is the number of $+$-edges in the cycle. By \eqref{Eq:Lapl}, $M = 
\sum_{1 \le p < q \le n} c_{pq}^-(1 - \sigma_{pq}) + c_{pq}^+(1 - 
\tau_{pq})$ where $\sigma_{pq}$ is as before and $\tau_{pq}$ is a 
reflection mapping $\tau_{pq}(u_p) = -u_q$, $\tau_{pq}(u_q) = -u_p$ and 
$\tau_{pq}(u_i) = u_i$ for $i \ne p,q$; the reflections $\sigma$ and $\tau$ 
generate a Coxeter group $D_n$. So one has

 \begin{corollary} \label{Cr:MTTD}
The characteristic polynomial of the operator 
 \begin{equation*}
M = \sum_{1 \le p < q \le n} c_{pq}^-(1 - \sigma_{pq}) + c_{pq}^+(1 - 
\tau_{pq})
 \end{equation*}
is equal to $\sum_{k=0}^n (-1)^k \mu_k t^k$ where 
 \begin{equation*}
\mu_k = \sum_{F \in \name{\mathcal{MFD}}_{n,k,\ell}} 
\prod_{\substack{(pq)_s \text{ is}\\ \text{an edge of } F}} c_{pq}^s 
\prod_{i=1}^{n-k} (m_i+1).
 \end{equation*}
Here $\name{\mathcal{MFD}}_{n,k,\ell}$ is the set of mixed forests with $k$ 
edges $(pq)_s$, $n-k$ tree components and $\ell$ one-cycle components such 
that the number of $+$-edges in every cycle is odd; $m_i$ is the number of 
edges in the $i$-th components.
 \end{corollary}

This corollary generalizes \cite[Theorem 3.2]{BPT}.

\subsubsection{The level $2$ Laplacian}\label{SSSec:Level2}

Take up the same setting as in Section \ref{SSSec:Lapl} (a line bundle with 
a connection over a graph $F$), and take
 \begin{equation}\label{Eq:Quadr}
P(x_{12}, \dots, x_{n-1,n}) = \sum_{\substack{1 \le i \le n, \\ 1 \le j < k 
\le n}} c_{ijk} (x_{ij} x_{ik} - x_{ik} x_{ij}) = \sum_{1 \le i,j,k \le n} 
c_{ijk} x_{ij} x_{ik};
 \end{equation}
here the constants $c_{ijk}$ are defined for all $1 \le i,j,k \le n$ and 
possess the property $c_{ijk} = -c_{ikj}$ for all $i,j,k$ (in particular, 
$c_{ijj} = 0$ for all $i, j$). Define $M$ by \eqref{Eq:DefM}: $M \bydef 
P(M_{12}, \dots, M_{n-1,n})$ where $M_{ij}(v) = \langle 
\alpha_{ij},v\rangle e_{ij}$. Explicitly, if $v = \sum_{p=1}^n v_p u_p$ 
then 
 \begin{equation*}
M_{ij}(v) = (v_i - \phi_{ji} v_j)(u_i - \phi_{ji} u_j) = (v_i - \phi_{ji} 
v_j) u_i + (v_j - \phi_{ij} v_i) u_j.
 \end{equation*}
and
 \begin{equation}\label{Eq:MExpl}
M(v) = \sum_{i \ne j} u_i v_j \left(\phi_{ij} \sum_{k \ne i,j} 
(c_{ijk}+c_{jki}) + \sum_{k \ne i,j} c_{kij} \phi_{ik} \phi_{kj}\right).
 \end{equation}

 \begin{remark} \label{Rm:Order}
Note that $M_{ij} = M_{ji}$ because $\alpha_{ji} = -\phi_{ji} \alpha_{ij}$ 
and $e_{ji} = -\phi_{ij} e_{ij}$. Nevertheless, since $\alpha$ and $e$ 
enter the equation \eqref{Eq:Main} separately, one has to choose the 
ordering of indices in every pair $(i,j)$ used; the final result, of 
course, does not depend on the choice. We will write $\alpha_{[ij]}$ 
meaning $\alpha_{ij}$ or $\alpha_{ji}$ depending on the choice; the same 
for $e$.
 \end{remark}

 \begin{remark}
In particular, if $\phi_{ij} = 1$ for all $i,j$, then $M(v) = \sum_{1 \le i 
< j < k \le n} w_{ijk} v_i u_j$ where $w_{ijk} \bydef 
c_{ijk}+c_{jki}+c_{kij}$. Note that $w_{ijk}$ is totally skew-symmetric in 
all the three indices, and therefore the operator $M$ is skew-symmetric. In 
\cite{MasVai} a beautiful formula for the Pfaffian of $M$ was proved; see 
also \cite{BPT} for a proof of the same formula using a technique close to 
Theorem \ref{Th:Main}.
 \end{remark}

We call $M$ the {\em level $2$ Laplacian} of the bundle, by an apparent 
analogy with the operator defined by \eqref{Eq:Lapl}. Note that $M_{ij}$ 
and $M_{kl}$ commute if $\{i,j\} \cap \{k,l\} = \emptyset$ or $\{i,j\} = 
\{k,l\}$, that's why $P$ contains no terms like $x_{ij} x_{kl} - x_{kl} 
x_{ij}$.

Application of Theorem \ref{Th:Main} to the level $2$ Laplacian gives the 
formula similar to \eqref{Eq:MixForest}, where the summation is done over 
the set of triangulated polyhedra of special kind.

A {\em nodal surface} is obtained from a smooth surface (a $2$-manifold, 
not necessarily connected, possibly with boundary), by gluing a finite 
number of points. If the surface has boundary then boundary points also can 
be glued. A boundary of the nodal surface is still well-defined, but unlike 
the boundary of a smooth surface, it can be any graph, not just collection 
of circles. Nodal surfaces with boundary attracted much attention in recent 
years due to their connection with the geometry of moduli spaces of complex 
structures, see \cite{Penner}.

Depending on the surface, we will speak about nodal disks, annuli, Moebius 
bands, etc.

A compact $2$-dimensional polyhedron $H$ is called reducible if it can be 
split into a union $H = H_1 \cup \dots \cup H_\ell$ where the polyhedra 
$H_i$ are compact and edge-disjoint (but not necessarily vertex-disjoint!). 
Every reducible polyhedron is a union of irreducible components. 

A $2$-dimensional polyhedron is called a {\em cycle polyhedron} if it is 
ireducible, homeomorphic to a nodal surface, and every its face is a 
triangle with exactly one side on the boundary. A $2$-dimensional 
polyhedron is called a {\em chain polyhedron}, if the last condition is 
satisfied for all the faces except two. Each of these two faces has two 
sides on the boundary. One face is called an initial face; one of its 
boundary sides is marked an called an initial side. The other exceptional 
face is called a terminal one; one of its boundary sides is marked and 
called a terminal side.

Choosing an orientation of a face of a cycle polyhedron is equivalent to 
ordering its sides: the first internal side, the second, the boundary side. 
For a chain polyhedron the rule is the same except for the initial and the 
terminal face. For the initial face the ordering is: the initial side, the 
internal side, the second boundary side; for the terminal face: the 
internal side, the terminal side, the second boundary side. Say that an 
orientation of two adjacent faces sharing a side $a$ is compatible if $a$ 
is the first internal side in one face and the second internal side in the 
other.

 \begin{lemma} \label{Lm:CycleChain}
A cycle polyhedron is one of the following:

 \begin{enumerate}
\item a nodal annulus where all vertices lie in the boundary;

\item a nodal Moebius band with the same property;

\item\label{It:Disk} a disk (smooth) with a vertex in the interior joined 
by edges with all the other vertices, which lie on the boundary.
 \end{enumerate}

A chain polyhedron is a nodal disk.

For every cycle polyhedron there are exactly two ways to orient all its 
faces compatibly. For every chain polyhedron (where the initial and 
terminal sides are chosen) there is exactly one way to orient all its faces 
compatibly.
 \end{lemma}

See Section \ref{SSec:Level2} for proof. 

 \begin{example}
See Fig.~\ref{Fg:Bristle}. The left-hand side is a cycle polyhedron (a 
nodal Moebius band with the node $c$), the right-hand side is a chain 
polyhedron (a nodal disk with the node $c$). Solid lines represent internal 
edges and exceptional edges ($ab$ and $ef$ at the nodal disk); dotted lines 
represent boundary edges. The graph $G$ is drawn below; the cycle is 
directed clockwise and a line is directed left to right. The corresponding 
ordering of edges inside every face of $H$ is shown by the numbers $1,2,3$.
 \end{example}

 \begin{figure}
\setlength{\unitlength}{0.075in}
\begin{picture}(36.85,22.85)
\put(33.05,15.33){{\setbox0=\hbox{{\tiny 3}}\lower\ht0\box0}}
\put(33.61,18.53){{\setbox0=\hbox{{\tiny 2}}\lower\ht0\box0}}
\put(32.15,16.36){{\setbox0=\hbox{{\tiny 1}}\lower\ht0\box0}}
\put(26.28,15.36){{\setbox0=\hbox{{\tiny 3}}\lower\ht0\box0}}
\put(25.25,17.90){{\setbox0=\hbox{{\tiny 2}}\lower\ht0\box0}}
\put(24.71,16.36){{\setbox0=\hbox{{\tiny 1}}\lower\ht0\box0}}
\put(12.58,15.26){{\setbox0=\hbox{{\tiny 3}}\lower\ht0\box0}}
\put(13.51,17.13){{\setbox0=\hbox{{\tiny 2}}\lower\ht0\box0}}
\put(12.01,16.03){{\setbox0=\hbox{{\tiny 1}}\lower\ht0\box0}}
\put(11.68,20.73){{\setbox0=\hbox{{\tiny 3}}\lower\ht0\box0}}
\put(12.68,19.53){{\setbox0=\hbox{{\tiny 2}}\lower\ht0\box0}}
\put(11.41,17.50){{\setbox0=\hbox{{\tiny 1}}\lower\ht0\box0}}
\put(8.08,20.80){{\setbox0=\hbox{{\tiny 3}}\lower\ht0\box0}}
\put(10.11,18.86){{\setbox0=\hbox{{\tiny 2}}\lower\ht0\box0}}
\put(6.38,20.00){{\setbox0=\hbox{{\tiny 1}}\lower\ht0\box0}}
\put(8.78,15.30){{\setbox0=\hbox{{\tiny 3}}\lower\ht0\box0}}
\put(8.15,16.83){{\setbox0=\hbox{{\tiny 2}}\lower\ht0\box0}}
\put(7.51,15.80){{\setbox0=\hbox{{\tiny 1}}\lower\ht0\box0}}
\put(5.45,15.33){{\setbox0=\hbox{{\tiny 3}}\lower\ht0\box0}}
\put(5.51,17.83){{\setbox0=\hbox{{\tiny 2}}\lower\ht0\box0}}
\put(4.75,16.83){{\setbox0=\hbox{{\tiny 1}}\lower\ht0\box0}}
\put(35.55,6.80){{\setbox0=\hbox{$(e,f)$}\lower\ht0\box0}}
\put(32.35,10.36){{\setbox0=\hbox{$(c,e)$}\lower\ht0\box0}}
\put(30.51,6.80){{\setbox0=\hbox{$(c,d)$}\lower\ht0\box0}}
\put(27.35,10.26){{\setbox0=\hbox{$(b,d)$}\lower\ht0\box0}}
\put(25.45,6.66){{\setbox0=\hbox{$(b,c)$}\lower\ht0\box0}}
\put(22.45,10.20){{\setbox0=\hbox{$(a,b)$}\lower\ht0\box0}}
\put(36.81,7.83){{\setbox0=\hbox{$\scriptstyle\bullet$}\kern-.4\wd0\lower.5\ht0\box0}}
\put(34.28,7.83){{\setbox0=\hbox{$\scriptstyle\bullet$}\kern-.4\wd0\lower.5\ht0\box0}}
\put(31.85,7.83){{\setbox0=\hbox{$\scriptstyle\bullet$}\kern-.4\wd0\lower.5\ht0\box0}}
\put(29.28,7.83){{\setbox0=\hbox{$\scriptstyle\bullet$}\kern-.4\wd0\lower.5\ht0\box0}}
\put(26.75,7.83){{\setbox0=\hbox{$\scriptstyle\bullet$}\kern-.4\wd0\lower.5\ht0\box0}}
\put(24.31,7.83){{\setbox0=\hbox{$\scriptstyle\bullet$}\kern-.4\wd0\lower.5\ht0\box0}}
\put(3.98,6.16){{\setbox0=\hbox{$\scriptstyle\bullet$}\kern-.4\wd0\lower.5\ht0\box0}}
\put(6.38,1.16){{\setbox0=\hbox{$\scriptstyle\bullet$}\kern-.4\wd0\lower.5\ht0\box0}}
\put(12.35,1.16){{\setbox0=\hbox{$\scriptstyle\bullet$}\kern-.4\wd0\lower.5\ht0\box0}}
\put(14.38,6.16){{\setbox0=\hbox{$\scriptstyle\bullet$}\kern-.4\wd0\lower.5\ht0\box0}}
\put(9.28,9.46){{\setbox0=\hbox{$\scriptstyle\bullet$}\kern-.4\wd0\lower.5\ht0\box0}}
\special{em:linewidth 0.014in}
\put(24.21,7.83){\special{em:moveto}}
\put(36.85,7.83){\special{em:lineto}}
\put(34.51,22.66){{\setbox0=\hbox{f}\lower\ht0\box0}}
\put(24.38,14.66){\special{em:moveto}}
\put(24.71,14.66){\special{em:lineto}}
\put(25.05,14.66){\special{em:moveto}}
\put(25.38,14.66){\special{em:lineto}}
\put(25.71,14.66){\special{em:moveto}}
\put(26.05,14.66){\special{em:lineto}}
\put(26.38,14.66){\special{em:moveto}}
\put(26.71,14.66){\special{em:lineto}}
\put(27.05,14.66){\special{em:moveto}}
\put(27.38,14.66){\special{em:lineto}}
\put(27.71,14.66){\special{em:moveto}}
\put(28.05,14.66){\special{em:lineto}}
\put(28.38,14.66){\special{em:moveto}}
\put(28.71,14.66){\special{em:lineto}}
\put(29.05,14.66){\special{em:moveto}}
\put(29.38,14.66){\special{em:lineto}}
\put(29.71,14.66){\special{em:moveto}}
\put(30.05,14.66){\special{em:lineto}}
\put(30.38,14.66){\special{em:moveto}}
\put(30.71,14.66){\special{em:lineto}}
\put(31.05,14.66){\special{em:moveto}}
\put(31.38,14.66){\special{em:lineto}}
\put(31.71,14.66){\special{em:moveto}}
\put(32.05,14.66){\special{em:lineto}}
\put(32.38,14.66){\special{em:moveto}}
\put(32.71,14.66){\special{em:lineto}}
\put(33.05,14.66){\special{em:moveto}}
\put(33.38,14.66){\special{em:lineto}}
\put(33.71,14.66){\special{em:moveto}}
\put(34.05,14.66){\special{em:lineto}}
\put(34.35,13.76){{\setbox0=\hbox{e}\lower\ht0\box0}}
\put(30.98,22.83){{\setbox0=\hbox{c}\lower\ht0\box0}}
\put(30.81,13.73){{\setbox0=\hbox{d}\lower\ht0\box0}}
\put(27.55,13.60){{\setbox0=\hbox{c}\lower\ht0\box0}}
\put(24.35,22.83){{\setbox0=\hbox{b}\lower\ht0\box0}}
\put(24.25,13.63){{\setbox0=\hbox{a}\lower\ht0\box0}}
\put(24.38,21.16){\special{em:moveto}}
\put(24.71,21.16){\special{em:lineto}}
\put(25.05,21.16){\special{em:moveto}}
\put(25.38,21.16){\special{em:lineto}}
\put(25.71,21.16){\special{em:moveto}}
\put(26.05,21.16){\special{em:lineto}}
\put(26.38,21.16){\special{em:moveto}}
\put(26.71,21.16){\special{em:lineto}}
\put(27.05,21.16){\special{em:moveto}}
\put(27.38,21.16){\special{em:lineto}}
\put(27.71,21.16){\special{em:moveto}}
\put(28.05,21.16){\special{em:lineto}}
\put(28.38,21.16){\special{em:moveto}}
\put(28.71,21.16){\special{em:lineto}}
\put(29.05,21.16){\special{em:moveto}}
\put(29.38,21.16){\special{em:lineto}}
\put(29.71,21.16){\special{em:moveto}}
\put(30.05,21.16){\special{em:lineto}}
\put(30.38,21.16){\special{em:moveto}}
\put(30.71,21.16){\special{em:lineto}}
\put(31.05,21.16){\special{em:moveto}}
\put(31.38,21.16){\special{em:lineto}}
\put(31.71,21.16){\special{em:moveto}}
\put(32.05,21.16){\special{em:lineto}}
\put(32.38,21.16){\special{em:moveto}}
\put(32.71,21.16){\special{em:lineto}}
\put(33.05,21.16){\special{em:moveto}}
\put(33.38,21.16){\special{em:lineto}}
\put(33.71,21.16){\special{em:moveto}}
\put(34.05,21.16){\special{em:lineto}}
\put(34.38,21.16){\special{em:moveto}}
\put(34.38,14.50){\special{em:lineto}}
\put(31.05,14.50){\special{em:moveto}}
\put(34.38,21.16){\special{em:lineto}}
\put(31.05,14.50){\special{em:moveto}}
\put(31.05,21.16){\special{em:lineto}}
\put(24.38,21.16){\special{em:moveto}}
\put(31.05,14.50){\special{em:lineto}}
\put(24.38,21.16){\special{em:moveto}}
\put(27.71,14.50){\special{em:lineto}}
\put(24.38,21.16){\special{em:moveto}}
\put(24.38,14.50){\special{em:lineto}}
\put(0.01,6.93){{\setbox0=\hbox{$(b,c)$}\lower\ht0\box0}}
\put(5.95,0.00){{\setbox0=\hbox{$(c,d)$}\lower\ht0\box0}}
\put(12.48,0.06){{\setbox0=\hbox{$(b,d)$}\lower\ht0\box0}}
\put(15.05,6.86){{\setbox0=\hbox{$(b,c)$}\lower\ht0\box0}}
\put(8.45,11.83){{\setbox0=\hbox{$(a,b)$}\lower\ht0\box0}}
\put(14.28,22.83){{\setbox0=\hbox{a}\lower\ht0\box0}}
\put(10.98,22.83){{\setbox0=\hbox{c}\lower\ht0\box0}}
\put(14.11,13.73){{\setbox0=\hbox{b}\lower\ht0\box0}}
\put(10.81,13.73){{\setbox0=\hbox{d}\lower\ht0\box0}}
\put(7.55,13.60){{\setbox0=\hbox{c}\lower\ht0\box0}}
\put(4.35,22.83){{\setbox0=\hbox{b}\lower\ht0\box0}}
\put(4.25,13.63){{\setbox0=\hbox{a}\lower\ht0\box0}}
\put(9.31,9.50){\special{em:moveto}}
\put(14.38,6.16){\special{em:lineto}}
\put(12.35,1.16){\special{em:lineto}}
\put(6.38,1.16){\special{em:lineto}}
\put(3.98,6.16){\special{em:lineto}}
\put(9.31,9.50){\special{em:lineto}}
\put(4.38,21.16){\special{em:moveto}}
\put(4.71,21.16){\special{em:lineto}}
\put(5.05,21.16){\special{em:moveto}}
\put(5.38,21.16){\special{em:lineto}}
\put(5.71,21.16){\special{em:moveto}}
\put(6.05,21.16){\special{em:lineto}}
\put(6.38,21.16){\special{em:moveto}}
\put(6.71,21.16){\special{em:lineto}}
\put(7.05,21.16){\special{em:moveto}}
\put(7.38,21.16){\special{em:lineto}}
\put(7.71,21.16){\special{em:moveto}}
\put(8.05,21.16){\special{em:lineto}}
\put(8.38,21.16){\special{em:moveto}}
\put(8.71,21.16){\special{em:lineto}}
\put(9.05,21.16){\special{em:moveto}}
\put(9.38,21.16){\special{em:lineto}}
\put(9.71,21.16){\special{em:moveto}}
\put(10.05,21.16){\special{em:lineto}}
\put(10.38,21.16){\special{em:moveto}}
\put(10.71,21.16){\special{em:lineto}}
\put(11.05,21.16){\special{em:moveto}}
\put(11.38,21.16){\special{em:lineto}}
\put(11.71,21.16){\special{em:moveto}}
\put(12.05,21.16){\special{em:lineto}}
\put(12.38,21.16){\special{em:moveto}}
\put(12.71,21.16){\special{em:lineto}}
\put(13.05,21.16){\special{em:moveto}}
\put(13.38,21.16){\special{em:lineto}}
\put(13.71,21.16){\special{em:moveto}}
\put(14.05,21.16){\special{em:lineto}}
\put(14.38,14.50){\special{em:moveto}}
\put(14.05,14.50){\special{em:lineto}}
\put(13.71,14.50){\special{em:moveto}}
\put(13.38,14.50){\special{em:lineto}}
\put(13.05,14.50){\special{em:moveto}}
\put(12.71,14.50){\special{em:lineto}}
\put(12.38,14.50){\special{em:moveto}}
\put(12.05,14.50){\special{em:lineto}}
\put(11.71,14.50){\special{em:moveto}}
\put(11.38,14.50){\special{em:lineto}}
\put(11.05,14.50){\special{em:moveto}}
\put(10.71,14.50){\special{em:lineto}}
\put(10.38,14.50){\special{em:moveto}}
\put(10.05,14.50){\special{em:lineto}}
\put(9.71,14.50){\special{em:moveto}}
\put(9.38,14.50){\special{em:lineto}}
\put(9.05,14.50){\special{em:moveto}}
\put(8.71,14.50){\special{em:lineto}}
\put(8.38,14.50){\special{em:moveto}}
\put(8.05,14.50){\special{em:lineto}}
\put(7.71,14.50){\special{em:moveto}}
\put(7.38,14.50){\special{em:lineto}}
\put(7.05,14.50){\special{em:moveto}}
\put(6.71,14.50){\special{em:lineto}}
\put(6.38,14.50){\special{em:moveto}}
\put(6.05,14.50){\special{em:lineto}}
\put(5.71,14.50){\special{em:moveto}}
\put(5.38,14.50){\special{em:lineto}}
\put(5.05,14.50){\special{em:moveto}}
\put(4.71,14.50){\special{em:lineto}}
\put(14.38,21.16){\special{em:moveto}}
\put(14.38,14.50){\special{em:lineto}}
\put(11.05,14.50){\special{em:moveto}}
\put(14.38,21.16){\special{em:lineto}}
\put(11.05,14.50){\special{em:moveto}}
\put(11.05,21.16){\special{em:lineto}}
\put(4.38,21.16){\special{em:moveto}}
\put(11.05,14.50){\special{em:lineto}}
\put(4.38,21.16){\special{em:moveto}}
\put(7.71,14.50){\special{em:lineto}}
\put(4.38,21.16){\special{em:moveto}}
\put(4.38,14.50){\special{em:lineto}}
\end{picture}

\caption{A cycle polyhedron and a chain polyhedron}\label{Fg:Bristle}
 \end{figure}
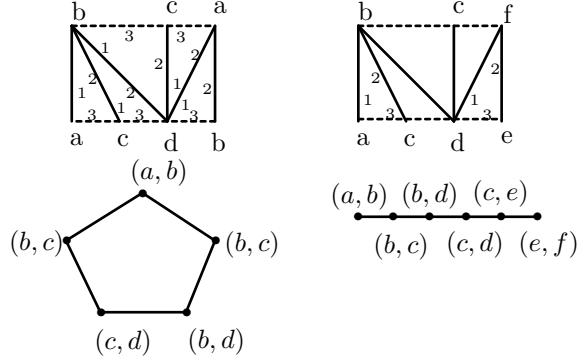

Denote by $\name{\mathcal{CP}}_{n,k}$ the set of polyhedra with the 
vertices $1, 2, \dots, n$, having $k$ faces, such that every its 
irreducible component is either a cycle polyhedron or a chain polyhedron 
with compatibly oriented faces. By $\name{ess}_1(H)$ denote the graph 
formed by the internal edges of $H$ and initial and terminal edges of its 
chain components. The initial and the terminal edges themselves denote by 
$I_1(H), \dots, I_s(H)$ and $T_1(H), \dots, T_s(H)$, respectively. Let $H 
\in \name{\mathcal{CP}}_{n,k}$ be such that the graphs $H^I \bydef 
\name{ess}_1(H) \setminus \{I_1(H), \dots, I_s(H)\}$ and $H^T \bydef 
\name{ess}_1(H) \setminus \{T_1(H), \dots, T_s(H)\}$ are mixed forests. Let 
$H^I_1, \dots, H^I_t$ and $H^T_1, \dots, 
H^T_t$ be connected components of $H^I$ and $H^T$, respectively, that are
trees (the number of such components $t = m-n-s$, where $m$ is the number 
of edges in $H$, is the same for both graphs). Choose in every tree a root 
$r_i^\alpha \in H^I_i$ and $r_j^e \in H^T_j$ and consider a matrix $M(H)$ 
such that 
 \begin{equation*}
M(H)_i^j = \sum_{\Lambda \in L_{ij}} \phi_\Lambda^2 n_\Lambda
 \end{equation*}
where $L_{ij}$ is the set of paths joining $r_j^e$ with $r_i^\alpha$, and 
$n_\Lambda$ is the number of vertices along the path $\Lambda$ that belong 
both to $H_i^I$ and $H_j^T$.

For an oriented face $F$ of a cycle or a chain polyhedron denote by 
$s_i(F)$ its $i$-th side ($i = 1, 2, 3$); by $v_i(F)$ denote the vertex 
opposite to the side number $i$. The internal sides are directed away from 
their common point; direction of the boundary side is not important.

 \begin{theorem} \label{Th:DetLevel2}
Let $M$ be a level $2$ Laplacian of a line bundle on a graph $F$. Then 
$\name{char}_M(t) = \sum_{k=0}^n (-1)^k \mu_k t^k$ where
 \begin{equation}\label{Eq:Level2}
 \begin{aligned}
\mu_k = \sum_{H \in \name{\mathcal{CP}}_{n,k}} \prod_{\substack{\Phi \text{ 
is}\\ \text{a face of }H}} &c_{v_3(\Phi)v_1(\Phi)v_2(\Phi)} \langle 
\alpha_{[s_1(\Phi)]}, e_{[s_2(\Phi)]}\rangle \\
&\times \det M(H) \times \prod_{i,j=1}^s \prod_{\substack{(pq) \in H \text{ 
lies}\\ \text{in a path joining}\\ r_i^\alpha \text{ with } r_j^e}} 
\phi_{pq}
 \end{aligned}
 \end{equation}
 \end{theorem}

See Remark \ref{Rm:Order} for the notation $\alpha_{[s_1]}$, $e_{[s_2]}$.

\section{Proofs}\label{Sec:Proofs}

\subsection{Technical lemmas}\label{SSec:Tech}

Let $H$ be a directed graph without loops or multiple edges. 

 \begin{Definition}[\cite{Kenyon}]
A {\em line bundle with connection on $H$} is a function a number 
$\phi_{ij} \ne 0$ to every directed edge $(i,j)$ of the graph. By 
definition, also take $\phi_{ji} = \phi_{ij}^{-1}$.
 \end{Definition}

To explain the name, attach a one-dimensional space $\Real$ (a fiber of 
the bundle) to every vertex and interpret the number $\phi_{ij}$ as the $1 
\times 1$-matrix of the operator of parallel transport along the edge 
$(i,j)$. Consider the space $\Real^n$ (which can be interpreted as a total 
space of the bundle) with the standard basis $u_1, \dots, u_n$. For every 
$i$, $j$ consider the vectors $\alpha_{ij} \bydef u_i - \phi_{ij} u_j$ and 
$e_{ij} \bydef u_i - \phi_{ji} u_j$. Denote by $A_H$ and $E_H$ the sets of 
vectors $\alpha_{ij}$ or $e_{ij}$, respectively, where $(i,j)$ runs through 
the edges of $H$.

Introduce in $\Real^n$ a standard scalar product making $u_1, \dots, u_n$ 
an orthonormal basis. For two sequences of vectors $M = (\mu_1, \dots, 
\mu_k)$ and $X = (\xi_1, \dots, \xi_k)$ of the same size $k$ denote by 
$G(M,X)$ the $k \times k$-matrix where the $(i,j)$-th entry is equal to the 
scalar product $(\mu_i,\xi_j)$. Call two sequences, $M$ and $M'$, 
elementarily equivalent, if $M'$ can be obtained from $M$ by a finite 
number of substitutions $\mu_i \mapsto \mu_i + t \mu_j$ where $t$ is any 
number. The matrices $G(M,X)$ and $G(M',X)$ are also connected by 
elementary equivalence (a row is replaced by the sum of itself with a 
multiple of another row), and therefore $\det G(M,X) = \det G(M',X)$; the 
same applies to $X$. 

Number the edges of $H$ arbitrarily and consider the matrices $P_H \bydef 
G(A_H,A_H)$, $Q_H = \bydef G(A_H,E_H)$, $R_H = \bydef G(E_H,E_H)$. The 
matrix element of $P_H$, $Q_H$ and $R_H$ corresponding to the pair of edges 
$s$ and $t$ depends on their mutual position as follows:
 \begin{itemize}
\item If $s$ and $t$ do not intersect then the matrix element is 
zero.

\item If $s$ and $t$ have one common vertex $b$ then the matrix element is 
equal to the product $\psi_{s,a} \psi_{t,a}$ where 
 \begin{equation*}
\psi_{r,a} = \begin{cases}
1, &\text{if the edge $r$ points away from the vertex $a$}, \\
-\phi_r, &\text{if the edge $r$ points towards the vertex $a$} \\
&\text{and corresponds to the vector $\alpha$},\\
-\phi_r^{-1}, &\text{if the edge $r$ points towards the vertex $a$}\\
&\text{and corresponds to the vector $e$}.
\end{cases}
 \end{equation*}

\item If the edges have two common vertices $a$ and $b$ then the product is 
equal to the sum $\psi_{s,a} \psi_{t,a} + \psi_{s,b} \psi_{t,b}$.
 \end{itemize}

The most important cases later will be when $H$ is a tree or a graph with 
one cycle Describe the determinants of $P_H$, $Q_H$ and $R_H$ for 
such $H$. For a directed path $\Lambda = (\lambda_0, \dots, \lambda_k)$ in 
the graph denote by $\phi_\Lambda$ the product $\phi_{\lambda_0 \lambda_1} 
\dots \phi_{\lambda_{k-1} \lambda_k}$.

 \begin{lemma} \label{Lm:DetAA}
Let $H$ be a rooted tree. Then 
 \begin{equation}\label{Eq:PTree}
\det P_H = \prod_{i,j} \phi_{ij}^2 \cdot \sum_\Lambda \phi_\Lambda^2
 \end{equation}
where the product is taken over all the edges of $H$ directed away from the 
root, and and the sum, over all chains $\Lambda$ joining vertices of the 
tree with the root. Similarly, 
 \begin{equation*}
\det R_H = \prod_{i,j} \phi_{ji}^2 \cdot \sum_\Lambda \phi_\Lambda^{-2}
 \end{equation*}
in the same notation.
 \end{lemma}

 \begin{proof}
Clearly, the statements for $P_H$ and $R_H$ are equivalent, so consider the 
$P_H$ case. Since $\alpha_{ji} = \phi_{ji} \alpha_{ij}$, changing the 
direction of an edge $(i,j) \mapsto (j,i)$ multiplies both sides of 
\eqref{Eq:PTree} by $\phi_{ji}^2$. So, it is enough to prove the lemma for 
trees where all the edges are directed away from the root.

Let $p$ be a hanging edge of $H$ and $q$, its parent. Then 
$(\alpha_p,\alpha_p) = 1 + \phi_p^2$, $(\alpha_p, \alpha_q) = -\phi_q$ 
(the edge $q$ is directed towards $p$, and $p$, away from $q$), and 
$(\alpha_p,\alpha_x) = 0$ for all the other edges $x$. Develop the 
determinant $\det P_H$ by the row $p$ and then by the row $q$; this gives a 
relation 
 \begin{equation}\label{Eq:Recurs}
\det P_H = (1 + 1/\phi_p^2) \det P_{H'} - 1/\phi_q^2 \cdot \det P_{H''}
 \end{equation}
where $H'$ and $H''$ are $H$ with $p$ deleted and $p$ and $q$ deleted, 
respectively.

Suppose by induction that the theorem is proved for $H'$ and 
$H''$. Multiply both sides of \eqref{Eq:Recurs} by $\prod_e \phi_e^2$ 
where $e$ runs through the set of edges of $\Gamma$. It gives
 \begin{equation*}
\prod_e \phi_e^2 \det P_H = \sum_{\Lambda'} \phi_{\Lambda'}^2 + 1/\phi_p^2 
\cdot \sum_{\Lambda'} \phi_{\Lambda'}^2 - 1/\phi_p^2 \sum_{\Lambda''} 
\phi_{\Lambda''}^2
 \end{equation*}
where $\Lambda'$, $\Lambda''$ are chains in $H'$ and $H''$, respectively, 
joining vertices with the root. The first sum is over all the chains in 
$\Gamma$ not containing $p$. The second sum is over all the chains 
containing $p$, plus the union of $p$ with a chain in $\Gamma''$ (because 
$p$ is hanging). The third sum cancels the second term.
 \end{proof}

 \begin{lemma} \label{Lm:QTree}
If $H$ is a tree with $m$ edges (and $m+1$ vertices) then $\det Q_H = m+1$. 
 \end{lemma}

 \begin{proof}
Let $(i,j)$ be an edge of $H$. The graph $H \setminus (i,j)$ is a union of 
two trees $H_1$ and $H_2$ containing vertices $i$ and $j$, respectively. 
The matrix $Q_H$ looks like 
 \begin{equation*}
Q_H = \left(\begin{array}{c|ccc|ccc}
2 & 1 & \dots & 1 & -\phi_{ji} & \dots & -\phi_{ji}\\
\hline
1 &  &  &  &  &  & \\
\vdots & & Q_{H_1} & & & 0 & \\
1 &  &  &  &  &  & \\
\hline
-\phi_{ij} &  &  &  &  &  & \\
\vdots & & 0 & & & Q_{H_2} & \\
-\phi_{ij} &  &  &  &  &  & 
\end{array}
\right)
 \end{equation*}
(we suppose that the edge $(i,j)$ corresponds to the first row and the 
first column). 

Denote the sizes of the first and the second block (i.e.\ the numbers of 
edges in $H_1$ and $H_2$) by $m_1$ and $m_2$ respectively (thus, $m = m_1 + 
m_2 + 1$). Let $p > m_1 > q$. Then the last $m_2$ rows of the minor 
$[Q_H]_{1,p+1}^{1,q+1}$ (deleted columns $1$ and $p+1$ and rows $1$ and 
$q+1$) have zeros at the $m_1$ initial positions; only the last $m - 2 - 
m_1 = m_2-1$ positions may be nonzero. Hence the rows are linearly 
dependent, so that $\det [Q_H]_{1,p+1}^{1,q+1} = 0$. The same is true if $p 
< m_1 < q$. Having $\det Q_H$ decomposed by the first row and then by the 
first column, one is left only with the terms $(Q_H)_{1,p+1} (Q_H)_{q+1,1} 
\det [Q_H]_{1,p+1}^{1,q+1}$ where either $p, q \le m_1$ or $p,q \ge m_1+1$. 
In both cases $(Q_H)_{1,p+1} (Q_H)_{q+1,1} = 1$, and therefore $\det Q_H$ 
does not depend on $\phi_{ij}$ (recall that $Q_{H_1}$ and $Q_{H_2}$ contain 
no $\phi_{ij}$ because neither $H_1$ nor $H_2$ have an edge $(i,j)$). Since 
$(i,j)$ is just an arbitrary edge of $H$, it proves that $\det Q_H$ does 
not depend on any $\phi_{pq}$ and is a constant.

Suppose now that $\phi_{pq} = 1$ for all $p$ and $q$. Then $e_{pr} = u_p - 
u_r = (u_p - u_q) + (u_q - u_r) = e_{pq} + e_{qr}$, and the same is true 
for $\alpha$. Suppose the tree $H$ contains edges $pq$ and $qr$. Consider a 
tree $H'$ where the edge $qr$ is replaced by $pr$; then the systems of 
vectors $A_H$ and $A_{H'}$, as well as $E_H$ and $E_{H'}$, differ by an 
elementary transformation, and therefore $\det Q_H = \det Q_{H'}$. By such 
operations one can convert $H$ into a line, i.e.\ a tree with the edges 
$(p_0, p_1), (p_1, p_2), \dots, (p_{m-1}, p_m)$. The matrix $Q_H$ for such 
tree is: $(Q_H)_{ii} = 2$, $(Q_H)_{i,i+1} = (Q_H)_{i,i-1} = 1$ for all $i$. 
An easy induction shows that $\det Q_H = m+1$.
 \end{proof}

Let now $H$ be a graph with one cycle, i.e.\ a connected graph with $n$ 
vertices and $n$ directed edges. It consists of a cycle $p_1, \dots, p_s$ 
and, possibly, some trees (``antlers'') attached to the vertices $p_i$. The 
direction of edges in the cycle and in the antlers can be arbitrary. 
Following \cite{Kenyon}, call the {\em holonomy} of the cycle the product 
$w \bydef \phi_{p_1 p_2}^{\pm 1} \dots \phi_{p_s p_1}^{\pm 1}$ where the 
$i$-th exponent is $+1$ if the corresponding edge is directed along the 
cycle (from $p_i$ to $p_{i+1}$) and $-1$ otherwise.

 \begin{lemma} \label{Lm:PQRCycle}
Let $H$ be a graph with one cycle. Then
 \begin{align*}
\det P_H &= (1-w)^2 \prod_{i,j} \phi_{ij}^2,\\
\det R_H &= (1-w^{-1})^2 \prod_{i,j} \phi_{ij}^{-2},\\
\det Q_H &= -(1-w)(1-w^{-1}),
 \end{align*}
where the product is taken over the set of all the edges in the antlers 
directed away from the cycle. 
 \end{lemma}

 \begin{proof}
The proofs of all the three statements are similar; here is the proof of 
the statement about $Q_H$.

Consider a system of vectors $E_H^{(1)}$ elementarily equivalent to $E_H$. 
To obtain $E_H^{(1)}$ from $E_H$ replace every vector $\eps_i \bydef 
e_{p_i p_{i+1}} = u_{p_i} - \phi_{p_{i+1} 
p_i} u_{p_{i+1}}$, $i = 1, \dots, s$, with $\eps_i^{(1)}$ where
 \begin{align*}
&\eps_{s-1}^{(1)} \bydef \eps_{s-1} + \phi_{p_s, p_{s-1}} 
\eps_s = u_{s-1} - \phi_{p_s, p_{s-1}} \phi_{p_1, 
p_s} u_1,\\
&\eps_{s-2}^{(1)} \bydef \eps_{s-2} + \phi_{p_{s-1}, p_{s-2}} 
\eps_{s-1}^{(1)} = u_{s-2} - \phi_{p_{s-1}, p_{s-2}} 
\phi_{p_s, p_{s-1}} \phi_{p_1, p_s} u_1,\\
&\vdots\\
&\eps_1^{(1)} = u_1 (1-\phi_{p_2 p_1} \dots \phi_{p_1, 
p_s}) = (1-1/w) u_1.
 \end{align*}
Consider also a system $A_H^{(1)}$ obtained from $A_H$ in a similar way. 
One has $\det G(E^{(1)},A_H^{(1)}) = Q_H$ by elementary equivalence.

If $w = 1$ then $\eps_1^{(1)} = 0$, so $\det Q_H = \det 
G(E^{(1)},A_H^{(1)}) = 0$, and the lemma is proved; we suppose from now on 
that $w \ne 1$.

Consider now the sequence $E_H^{(2)}$, which is $E_H^{(1)}$ with every 
$\eps_i^{(1)}$ replaced with
 \begin{equation*}
\eps_i^{(2)} = \eps_i^{(1)} - \eps_1^{(1)}/(1-1/w) = -\phi_{p_i,p_{i-1}} 
u_i.
 \end{equation*}
Consider also $A_H^{(2)}$ obtained from $A_H^{(1)}$ in a similar way. The 
sequences $E_H^{(2)}$, $A_H^{(2)}$ are elementarily equivalent to 
$E_H^{(1)}$, $A_H^{(1)}$, so $\det G(E_H^{(2)}, A_H^{(2)}) = Q_H$.

Denote by $H_i$ a subtree of $H$ attached to the vertex $p_i$, $1 \le i \le 
s$, and let $\beta_j = e_{q_j r_j}$, $1 \le j \le m_i$ be vectors in $E_H$ 
(and $E_H^{(2)}$) corresponding to its edges. For all edges attached 
immediately to $p_i$ (so that $q_j = p_i$) replace $\beta_j \mapsto 
\beta_j^{(1)} = \beta_j + \eps_i \phi_{p_{i+1},p_i} = -\phi_{r_j p_1} 
u_{r_j}$. Then do the same for all the edges attached to endpoints of 
$\beta_j^{(1)}$, etc. Having done this for all $i$, $1 \le i \le s$, one 
obtaines the system $E_H' = ((1-1/w)u_1, -\phi_{12} u_2, \dots, -\phi_{1m} 
u_m)$ elementarily equivalent to $E_H$. Similarly, $A_H$ is elementarily 
equivalent to $A_H' = ((1-w)u_1, -\phi_{12} u_2, \dots, -\phi_{m1} u_m)$. 
Now $\det Q_H = \det G(E_H',A_H')$; the matrix $G(E_H',A_H')$ is a diagonal 
matrix with $(1-w)(1-1/w)$ in the corner and $1$ in all the other positions 
on the main diagonal.
 \end{proof}

Below we will need two more statements from the general linear algebra. 

 \begin{lemma} \label{Lm:Gram}
Let $e_1, \dots, e_n$ be an orthonormal basis in $\Real^n$, and $\mu_i = 
\sum_{j=1}^n a_{ij} e_j$. Then $\det G(M,M)$ is equal to the sum of squares 
of all the $k \times k$-minors of the matrix $A = (a_{ij})_{1 \le i \le 
k}^{1 \le j \le n}$.
 \end{lemma}

The lemma is classical, see e.g.\ \cite[IX\S5]{Gantmacher} for proof.

Let now $\mathcal M, \mathcal X \subset \Real^n$ be two linear subspaces of 
the same dimension $k$, and $M = (\mu_1, \dots, \mu_k)$ and $X = (\xi_1, 
\dots, \xi_k)$ be bases in them. Denote 
 \begin{equation*}
\angle(\mathcal M, \mathcal X) \bydef \det G(M,X)^2/(\det G(M,M) \cdot \det 
G(X,X)).
 \end{equation*}

\def \M {\mathcal M}
\def \X {\mathcal X}

 \begin{lemma} \label{Lm:OrthComp}
 \begin{enumerate}
\item $\angle(\M,\X)$ depends only on the subspaces and not 
on the choice of the bases $M$ and $X$ in them. 

\item $\angle(\M^\perp,\X^\perp) = \angle(\M,\X)$. 
 \end{enumerate}
 \end{lemma}

 \begin{proof}
Let $N = (\nu_1, \dots, \nu_k)$ be another basis in $\M$; denote by $A = 
(a_{ij})$ the transfer matrix: $\nu_i = \sum_{j=1}^k a_{ij} \mu_j$. Then 
$G(N,X) = A G(M,X)$, and $G(N,N) = A G(M,M) A^*$; so $G(N,X)^2/(\det G(N,N) 
\cdot \det G(X,X)) = \det A^2 \det G(M,X)^2/(\det A^2 \cdot \det G(M,M) 
\cdot \det G(X,X)) = \angle(\M,\X)$; the same is for $X$.

Let now $\M_1, \M_2 \subset \Real^n$ be two spaces of equal dimension, and 
let $\M \bydef \langle \M_1 \cup \M_2\rangle$ be their linear hull. Then 
$\M_i^\perp = \M_i^{\perp,\M} \oplus \M^\perp$, $i = 1,2$; here $\perp$ means 
an orthogonal complement in $\Real^n$, and $(\perp,\M)$ means an orthogonal 
complement in $\M$. Choose an orthonormal basis $T = (\tau_1, \dots, 
\tau_q)$ in $\M^\perp$, and bases $\Lambda_i = (\lambda_1^{(i)}, \dots, 
\lambda_p^{(i)})$ in $\M_i^{\perp,\M}$, $i = 1,2$, normal to $\M^\perp$. 
Denote $Y_1 = (\lambda_1^{(1)}, \dots, \lambda_p^{(1)}, \tau_1, \dots, 
\tau_q)$ and $Y_2 = (\lambda_1^{(2)}, \dots, \lambda_p^{(2)}, \tau_1, 
\dots, \tau_q)$ are bases in $\M_1^\perp$ and $\M_2^\perp$, respectively. By 
the first statement of the theorem, $\angle(\M_1^\perp,\M_2^\perp) = \det 
G(Y_1,Y_2)^2/(\det G(Y_1.Y_1) \cdot \det G(Y_2,Y_2))$. The matrix 
$G(Y_1,Y_2)$ is block diagonal; its first block is the matrix 
$G(\Lambda_1,\Lambda_2)$, and the second block is the unit matrix $G(T,T)$; 
thus $\det G(Y_1,Y_2) = \det G(\Lambda_1,\Lambda_2)$. Similarly, $\det 
G(Y_1,Y_1) = \det G(\Lambda_1,\Lambda_1)$ and $\det G(Y_2,Y_2) = \det 
G(\Lambda_2,\Lambda_2)$. Hence, $\angle(\M_1^\perp,\M_2^\perp) = 
\angle(\M_1^{\perp,\M},\M_2^{\perp,\M})$. 

Let now $\X = \M_1 \cap \M_2 \ne 0$. A similar choice of a basis shows that 
$\angle(\M_1,\M_2) = \angle(\X_1,\X_2)$ where $\X_i = \X^{\perp,\M_i}$, $i = 1,2$. 
So, it is enough to prove the second statement of the theorem for the 
$k$-dimensional subspaces $\M_1, \M_2 \subset \Real^n$ such that $n = 2k$ and 
$\M_1 \cap \M_2 = 0$, so that $\Real^n = \langle \M_1 \cap \M_2\rangle$. To 
do this take an orthonormal basis $e_1, \dots, e_{2k}$ in $\Real^n$ such 
that $X_1 \bydef (e_1, \dots, e_k)$ is a basis in $\M_1$, and $X_2 \bydef 
(e_{k+1}, \dots, e_{2k})$ is a basis in $\M_1^\perp$. Let $F = (f_1, \dots, 
f_k)$ be a basis in $\M_2$, $f_i = \sum_{j=1}^k b_{ij} e_j + c_{ij} 
e_{j+k}$. Since $\M_1 \cap \M_2 = 0$, the matrix $C = (c_{ij})_{1 \le i \le 
k}^{1 \le j \le k}$ is nondegenerate. By elementary transformations of 
rows one can make $C$ a unit matrix; assume that $C = \name{id}$ from the 
very beginning, so that $f_i = \sum_{j=1}^k b_{ij} e_j + e_{i+k}$. Then 
$G(F,X_1) = B$, and therefore $\det G(F,X_1) = \Delta \bydef \det B$. By 
Lemma \ref{Lm:Gram}, the determinant $\det G(F,F)$ is equal to the sum of 
squares of all $k \times k$-minors of the $k \times (2k)$-matrix composed 
of two blocks, $B$ and the identity $k \times k$-matrix. It is easy to see 
that the latter sum is the sum of squares of all the minors of $B$ (of all 
sizes), including $\Delta^2$ and $1$ (the square of the empty minor). 

Consider now the vectors $h_i = -e_i + \sum_{j=1}^k b_{ji} e_{j+k}$, $i = 
1, \dots, k$. Apparently, $H = (h_1, \dots, h_k)$ is a basis in 
$\M_2^\perp$; also $G(H,X_2) = B^T = G(F,X_1)^T$, and also $\det G(H,H) = 
\det G(F,F)$ by Lemma \ref{Lm:Gram}. This proves the lemma.
 \end{proof}

\subsection{Proof of Theorem \ref{Th:Forman}}\label{SSec:MTree}

To apply Theorem \ref{Th:Main} note that the polynomial $P$ of 
\eqref{Eq:DefPLapl} has degree $1$, so every path involved should contain 
one vertex only. Therefore, the DOOMB $G$ must be a union of $k$ loops 
attached to the vertices $(i_1,j_1), \dots, (i_k,j_k)$. So the summation is 
over the set of unordered $k$-tuples $(i_1,j_1), \dots, (i_k,j_k)$ with $1 
\le i_s, j_s \le N$ for all $s$. In other words, the summation is over the 
set of graphs $F$ with the vertices $1, 2, \dots, n$ and $k$ unnumbered 
directed edges. 

Let $F_1, \dots, F_\ell$ be connected components of $F$. If the edges $d_1$ 
and $d_2$ belong to different components then $\langle \alpha_{d_1}, 
e_{d_2}\rangle = 0$. So the matrix $Q_F$ is block diagonal, and $\det Q_F = 
\det Q_{F_1} \dots \det Q_{F_\ell}$. If $F$ is a connected graph with $m$ 
edges $(i_1,j_1), \dots, (i_m,j_m)$ and $\mu < m$ vertices then the vectors 
$e_{i_1,j_1}, \dots, e_{i_m,j_m}$ belong to a vector space of dimension 
$\mu$ spanned by the corresponding basis elements $u_i$. Therefore they are 
linearly dependent, so that $\det Q_F = 0$. Thus, if $\det Q_F \ne 0$ then 
every connected component $F_i$ of $F$ should be either a tree (with $\mu = 
m+1$) or a graph with one cycle ($\mu = m$). So, $F$ is a mixed forest.

Theorem \ref{Th:Forman} now follows from Theorem \ref{Th:Main}, Lemma 
\ref{Lm:QTree} and Lemma \ref{Lm:PQRCycle}.

\subsection{Proof of Theorem \ref{Th:DetLevel2}}\label{SSec:Level2}

Apply Theorem \ref{Th:Main} to the operator $M$ of \eqref{Eq:MExpl}. Like 
in Section \ref{SSec:MTree}, vertices of the graph $G$ are pairs $(i,j)$, 
$1 \le i < j \le n$, that is, edges of a directed graph with the vertices 
$1, \dots, n$. The polynomial $P$ contains only terms $x_{ij} x_{ik}$; 
therefore if $(a,b)$ is an edge of $G$ then the pairs $a = \{i,j\}$ and $b 
= \{i,k\}$ have exactly one common element. Represent such edge by a 
triangle $ijk$. Color its sides $ij$ and $ik$ (corresponding to the vertices 
of $G$) black, and the third side $jk$, red (they are shown as sold and the 
dashed lines, respectively, in Fig.~\ref{Fg:Bristle}). Thus a graph $G$ is 
represented by $2$-dimensional polyhedron (call it $H$) with triangular 
faces and edges colored red and black so that every face has two black 
sides and one red side. The black edges of $H$ form a graph denoted by 
$\name{ess}_1(H)$ and called an essential $1$-skeleton of $H$. 

If $G$ consists of connected components $G_1, \dots, G_s$ then the 
corresponding subpolyhedra $H_1, \dots, H_s$ of $H$ are edge-disjoint but 
not necessarily vertex-disjoint; thus, $H$ is reducible, and $H_1, \dots, 
H_s$ are its irreducible components. By Theorem \ref{Th:Main} every $G_i$ 
is either an oriented cycle or an oriented chain. 

Let first $G_i$ be an oriented cycle. In the corresponding $H_i$ every 
black edge belongs to two faces and every red edge, to one face. Hence 
$H_i$ is a nodal surface with boundary where every face has two internal 
sides and one boundary side. An orientation of an edge of $G_i$ gives rise 
to an orientation of the corresponding face of $H_i$; since the 
orientations of the edges in a cycle $G_i$ are compatible, so are 
orientations of the faces in $H_i$. The graph $G_i$ is connected, so the 
polyhedron $H_i$ is irreducible. Hence, $H_i$ is a cycle polyhedron. 
Conversely, if $H_i$ is a cycle polyhedron, then $G_i$ is connected and 
every its vertex has valency $2$ --- hence, $G_i$ is a cycle.

In a similar manner one proves that $G_i$ is an oriented chain if and only 
if $H_i$ is a chain polyhedron.

 \begin{proof}[of Lemma \ref{Lm:CycleChain}]
Let $H$ be a cycle polyhedron with the vertices $1, 2, \dots, n$. Build the 
graph $G$ such that the vertices of $G$ are internal edges of $H$ (elements 
of $\name{ess}_1(H)$), and two vertices are joined by an edge if the 
corresponding edges belong to a face. 

As proved before, the graph $G$ is an oriented cycle where every vertex is 
marked by two indices $(i,j)$, and the neighbouring vertices have exactly 
one common index. For every index $i$ denote by $K_i$ the set of vertices 
in $G$ containing $i$ as one of the indices. If for some $i$ the 
corresponding $K_i$ contains all the vertices, then the second indices $j$ 
in all the vertices are pairwise distinct, and we have a disk described in 
clause \ref{It:Disk} of the lemma. From now on suppose that for every index 
$i$ there is at least one vertex of $G$ {\em not} containing it.

Every $K_i$ is a union of several ``solid arcs'' $K_{i,1}, \dots, K_{i,s_i}$ 
(a solid arc is one vertex or a sequence of several successive vertices 
in a cycle). Consider an auxiliary graph $G'$ where for all $i$ and 
$p = 1, \dots, s_i$ the index $i$ in the vertices in $K_{i,p}$ is renamed 
into a new index $i_p$. Thus, the polyhedron $H$ is obtained from the 
polyhedron $H'$ corresponding to $G'$ by identification of some vertices. 
Thus it is enough to prove that if every $K_i$ is one solid arc (but not 
the whole cycle) then the polyhedron $H$ is an annulus or a Moebius band.

Let $G$ contain $m$ vertices. Consider an auxiliary circle $S^1$ with $m$ 
equidistant points $a_1, \dots, a_m$ on it. Let $K_i$ be a solid arc 
covering the vertices $p, p+1, \dots, p+q$. Define a subset $L_i \subset 
S^1$ as an arc from $a_p$ to $a_{p+q}$, including both endpoints. Every two 
neighbouring vertices in $G$ have a common index $i$, so every point of 
every arc $[a_p,a_{p+1}]$ belongs to some $L_i$, hence $\bigcup L_i = S^1$. 
Triple intersections of different $L_i$s are all empty; the intersections 
$L_i \cap L_j$ are either empty or one point. There are exactly $m$ pairs 
with nonempty intersection, because $L_i$ are arcs. Thus, the Euler 
characteristics is
 \begin{equation*}
0 = \chi(S^1) = \chi(\bigcup L_i) = \sum_i \chi(L_i) - \sum_{i,j} \chi(L_i 
\cap L_j).
 \end{equation*}
One has $\chi(L_i) = 1$ for all $i$ and $\chi(L_i \cap L_j) = 1$ for $m$ 
pairs $i,j$ where the intersection is nonempty, and $\chi(L_i \cap L_j) = 
0$ otherwise. So, the number of indices $i$, that is, the number of 
vertices of the polyhedron $H$, is equal to $m$. The total number of its 
faces is the number of edges in $G$, that is, $m$. Also $H$ contains $m$ 
red edges (one per face) and $m$ black edges (corresponding to the 
vertices of $G$). Thus $\chi(H) = m-2m+m = 0$, and $H$ is either an annulus 
or a Moebius band.

The proof in the chain case is similar.
 \end{proof}

So, summation in the right-hand side of Theorem \ref{Th:Main} for the 
$\name{char}_M(t)$ is performed over the set of nodal sufraces with 
boundary $H = H_1 \cup \dots \cup H_m$ where the irreducible components 
$H_1, \dots, H_s$ are either cycle polyhedra or chain polyherda. 

The weight $W_P(H)$ of a polyhedron is equal to the product of the weights 
of the faces. The weight of a triangular face $pqr$ where $pq$ is the 
first internal edge, $pr$, the second, and $qr$, a boundary edge, is equal 
to $c_{pqr} \langle \alpha_{[pq]}, e_{[pr]}\rangle$. So, the second factor 
is equal to $1$, $-\phi_{pq}$, $-\phi_{pr}$ or $\phi_{pq} \phi_{pr}$ 
depending on how the edges $pq$ and $pr$ are directed. The weight of $H$ is 
$W_P(H) = W_P(H_1) \dots W_P(H_m)$. 

Let $H_1, \dots, H_s$ be chain polyhedra with the initial edges $I_1(H), 
\dots, I_s(H)$ and the terminal edges $T_1(H), \dots, T_s(H)$, and 
$H_{s+1}, \dots, H_m$ be cycle polyhedra. The determinant in 
\eqref{Eq:Main} is equal to $\det G(A_{H^I},E_{H^T})$ where $H^I \bydef 
\name{ess}_1(H) \setminus \{I_1(H), \dots, I_s(H)\}$ and $H^T \bydef 
\name{ess}_1(H) \setminus \{T_1(H), \dots, T_s(H)\}$. In particular, if the 
determinant is nonzero, then $e_i, i \ne I_1(H), \dots, I_s(H)$, as well as 
$\alpha_j, j \ne T_1(H), \dots, T_s(H)$, are linearly independent. 
According to Section \ref{SSec:MTree}, this implies that the graphs $H^I$ 
and $H^T$ are mixed forests. 

Thus, equation \eqref{Eq:Main} for the case considered looks as follows: if 
$M$ is the level 2 Laplacian then $\name{char}_M(t) = \sum_{k=0}^n (-1)^k 
\mu_k t^k$ where
 \begin{equation}\label{Eq:Level2Interm}
\mu_k = \sum_{\substack{H \in \name{\mathcal{CP}}_{n,k}\\H^I \text{ and } 
H^T \text{ are}\\ \text{mixed forests}}} \prod_{\substack{\Phi \text{ is}\\ 
\text{a face of }H, }} c_{v_3(\Phi)v_1(\Phi)v_2(\Phi)} \langle 
\alpha_{[s_1(\Phi)]}, e_{[s_2(\Phi)]}\rangle \times \det G(A_{H^I},E_{H^T})
 \end{equation}

The determinantal term in \eqref{Eq:Level2Interm} can be simplified using 
Lemma \ref{Lm:OrthComp}:
 \begin{align}
(\det G(A_{H^I},E_{H^T}))^2 &= \angle(\langle A_{H^I}\rangle, \langle 
E_{H^T}\rangle) \cdot \det G(A_{H^I},A_{H^I}) \det 
(E_{H^T},E_{H^T})\nonumber \\ 
&= \angle(\langle A_{H^I}\rangle^\perp, \langle E_{H^T}\rangle^\perp) \cdot 
\det G(A_{H^I},A_{H^I}) \det (E_{H^T},E_{H^T})\nonumber\\
&\begin{aligned}
&=\frac{\det G(A_{H^I}',E_{H^T}')^2}{\det G(A_{H^I}',A_{H^I}') \det 
G(E_{H^T}',E_{H^T}')} \\
&\hphantom{\det G(A_{H^I}',A_{H^I}')}\times \det G(A_{H^I},A_{H^I}) \det 
(E_{H^T},E_{H^T});
 \end{aligned}\label{Eq:Dual}
 \end{align}
here $\langle X\rangle$ mean the subspace in $\Real^n$ spanned by $X$, 
${}^\perp$ means an orthogonal complement, and $A_{H^I}'$, $E_{H^T}'$ are 
bases in $\langle A_{H^I}\rangle^\perp$ and $\langle E_{H^T}\rangle^\perp$, 
respectively.

By Lemma \ref{Lm:OrthComp} the formula above is true for any choice of the 
bases $A_{H^I}'$, $E_{H^T}'$; below we describe orthogonal bases the most 
convenient for our purposes. For any graph $G$ denote by $\name{\mathcal 
C}(G)$ the set of its connected components. Let $\name{\mathcal C}(H^I) = 
\{H^I_1, \dots, H^I_s\}$ and $\name{\mathcal C}(H^T) = \{H^T_1, \dots, 
H^T_t\}$; denote by $V_i^I \subset \Real^n$ and $V_i^T \subset \Real^n$ the 
subspaces spanned by the vertices of $H^I_i$ and $H^T_i$, respectively. 
Then 
 \begin{equation*}
\langle A_{H^I}\rangle^\perp = \bigoplus_{i=1}^s \langle 
A_{H^I_i}\rangle^{\perp,V_i^I},
 \end{equation*}
where the summands are pairwise orthogonal; the same is true for $E_{H^T}$.

Every $H^I_i$ and $H^T_i$ is either a tree or a graph with one cycle; 
choose a root in every tree component and denote it by $r_i^\alpha$ and 
$r_i^e$, respectively.

 \begin{lemma} \label{Lm:OCAlpha}
If $H$ is a tree with a root $r$, then the spaces $\langle 
A_H\rangle^\perp$ and $\langle E_H\rangle^\perp$ have dimension $1$ and are 
spanned by vectors $b_H \bydef \sum_u u\phi_{ur}$ and $f_H \bydef \sum_u 
u\phi_{ru}$, respectively, where the summation is over the set of vertices 
of $H$, and $ur$, $ru$ are paths joining $u$ with $r$ and $r$ with $u$.

If $H$ is a graph with one cycle with monodromy $w \ne 1$ then $\langle 
A_H\rangle = \langle E_H\rangle = \Real^n$.
 \end{lemma}

 \begin{proof}
Let $H$ be a tree. The spaces $\langle A_H\rangle$ and $\langle E_H\rangle$ 
are spanned by $n-1$ vectors in $\Real^n$; the vectors are linearly 
independent by Lemma \ref{Lm:QTree}. Hence, $\dim \langle A_H\rangle^\perp 
= \dim \langle E_H\rangle^\perp = 1$. Apparently, $(b_H,\alpha_{ij}) = 0$ 
and $(f_H, e_{ij}) = 0$ for any edge $ij$ of $H$, and the first statement 
follows.

The second statement follows from Lemma \ref{Lm:PQRCycle}.
 \end{proof}

The number of rows and columns of the matrix $G(A_{H^I}',E_{H^T}')$ is 
equal to the number of tree components of the graphs $H^I$ and $H^T$. If 
$H_i^I$ is a tree component of $H^I$ and $H_j^T$ is a tree component of 
$H^T$, then, obviously, $G(A_{H^I}',E_{H^T}')_i^j = (b_{H_i^I}, f_{H_j^T}) 
= \sum_{\Lambda \in L_{ij}} \phi_\Lambda^2 n_\Lambda = M(H)_i^j$ (recall 
that $L_{ij}$ is the set of paths joining $r_j^e$ with $r_i^\alpha$ and 
$n_\Lambda$ is the number of vertices along the path $\Lambda$ that belong 
both to $H_i^I$ and $H_j^T$).

Lemmas \ref{Lm:OCAlpha} and \ref{Lm:DetAA} imply now that
 \begin{align*}
\det G(A_{H^I},E_{H^T}) &= \det M(H)
\times \prod_{i=1}^s \prod_{\substack{(pq) \in H^I_i \text{ looks}\\ 
\text{away from } r_i^\alpha}} \phi_{pq} \times \prod_{j=1}^s 
\prod_{\substack{(pq) \in H^T_j \text{ looks}\\ \text{away from } 
r_j^e}} \phi_{qp} \\
&= \det M(H) \times \prod_{i,j=1}^s \prod_{\substack{(pq) \in H \text{ 
lies}\\ \text{in a path joining}\\ r_i^\alpha \text{ with } r_j^e}} 
\phi_{pq},
 \end{align*}
and Theorem \ref{Th:DetLevel2} is proved.

\section*{Acknowledgements}

During the work the autor was generously supported by the RFBR grants 
08-01-00110-  ``Geometry and combinatorics of mapping spaces for real and 
complex curves'' and NSh-4850.2012.1 ``Singularities theory and its 
applications'', by the Higher School of Economics (HSE) Scientific 
foundation grant 10-01-0029 ``Graph embeddings and integrable 
hierarchies'', by the joint RFBR ans HSE grant 09-01-12185-ofi-m 
``Combinatorial aspects of integrable models in mathematical physics'', HSE 
Laboratory of mathematical research TZ 62.0 project ``Math research in 
low-dimensional topology, algebraic geometry and representation theory'', 
and by the FTsP-Kadry grant ``New methods of research of integrable systems 
and moduli spaces in geometry, topology and mathematical physics''.

I would like to thank A.Abdessalam and A.Bufetov for stimulating 
discussions and for introducing me to great examples of research in the 
field.


\begin{thebibliography}{9}

\bibitem{Abdes} A.Abdessalam, ``The Grassmann-Berezin calculus and 
theorems of the matrix-tree type'', \textit{Advances in Applied 
Mathematics}, 33, no.~1 (2004): 51--70.

\bibitem{Borodin} A.Borodin, ``Determinantal point processes'', in {\em 
The Oxford handbook of random matrix theory}, edited by Gernot Akemann, 
Jinho Baik, and Philippe Di Francesco, Oxford Univesity Press, 2011.

\bibitem{BPT} Yu.Burman, A.Ploskonosov, A.Trofimova, {\em Higher 
matrix-tree theorems}, preprint arXiv:1109.6625 (math.CO). 

\bibitem{BZ} Yu.Burman, D.Zvonkine, ``Cycle factorizations and 1-faced 
graph embeddings'', \textit{European Journal of Combinatorics}, 31, no.~1 
(2010): 129--144. 

\bibitem{BurPem} R.Burton, R.Pemantle, ``Local characteristics, entropy and 
limit theorems for spanning trees and domino tilings via 
transfer-impedances'', \textit{Annals of Probability}, 21, no.~3 (1993): 
1329--1371.

\bibitem{Chaiken} S.\,Chaiken, ``A combinatorial proof of the all minors 
matrix tree theorem'', \textit{SIAM J.\ Alg.\ Disc.\ Meth.}, 3, no.~3 
(1982): 319--329.

\bibitem{Forman} R.\,Forman, ``Determinants of Laplacians on graphs'', 
\textit{Topology} 32, no.~1 (1993): 35--46.

\bibitem{Gantmacher} F.R.Gantmacher, {\em The theory of matrices}, AMS 
Chelsea publishing, 2000.

\bibitem{Kenyon} R.\,Kenyon, {\em The Laplacian on planar graphs and graphs 
on surfaces}, preprint arXiv:1203.1256 (math.PR, math.CO).

\bibitem{Kirchhoff} G.\,Kirchhoff, ``\"Uber die Aufl\"osung der 
Gleichungen, auf welche man bei der Untersuchung det linearen Verteilung 
galvanischer Str\"ome gefurht wird'', \textit{Ann.\ Phys.\ Chem.}, 72 
(1847): 497--508.  

\bibitem{MasVai} G.\,Masbaum, A.\,Vaintrob, ``A new matrix-tree theorem'', 
\textit{Internat.\ Math.\ Res.\ Notices}, 27 (2002): 1397--1426. 

\bibitem{Penner} D.J.\,LaFountain, R.C.\,Penner, {\em Cell decomposition 
and odd cycles on compactified Riemann's moduli space}, preprint 
arXiv:1112.3915 (math.GT, math.AG).

\end{thebibliography}
\end{document}